\documentclass[reqno]{amsart}
\newcommand \datum {May 26, 2021}

\usepackage{amssymb,latexsym}
\usepackage{amsmath}
\usepackage{amscd}
\usepackage{graphicx}
 \usepackage{color}
\usepackage{enumerate}
\numberwithin{equation}{section}
\theoremstyle{plain}
 \newtheorem{theorem}{Theorem}[section]
 
 \newtheorem{proposition}[theorem]{Proposition}
 \newtheorem{corollary}[theorem]{Corollary}
\theoremstyle{definition}
 \newtheorem{definition}[theorem]{Definition}

\newenvironment{enumeratei}{\begin{enumerate}[\quad\upshape (i)]} {\end{enumerate}}
\newcommand \tbf[1]  {\textbf{#1}} 
\newcommand \alg[1]  {\mathcal #1}

\newcommand \Nnul {\mathbb N_0}
\newcommand \Jir [1] {\textup J(#1)} 
\newcommand \Mir [1] {\textup M(#1)} 
\newcommand \Nplu {\mathbb N^+}
\newcommand \set [1]{\{#1\}}
\newcommand \id [1] {\textup{id}_{#1}}
\newcommand \sps {\mathcal{S}}
\newcommand \oisps {\mathcal{S}_{\bf{01}}}
\newcommand \lensps {\mathcal{S}_{\textup{\bf{len}}}}
\newcommand \zo {$\set{0,1}$-preserving}
\newcommand \lemb {length-preserving embedding}
\renewcommand \phi{\varphi}
\newcommand \mref[1] {\textup{(M\ref{#1})}}
\newcommand \lc [1]  {w_{\textup{left}}(#1)}
\newcommand \rc [1]  {w_{\textup{right}}(#1)}

\newcommand \ucov [1] {#1^+}
\newcommand \lcov [1] {#1^-}
\newcommand \lbound[1] {B_{\textup{left}}(#1)}
\newcommand \rbound[1] {B_{\textup{right}}(#1)}
\newcommand \Bnd[1] {\textup{Bnd}(#1)}
\newcommand \restrict [2] {#1\rceil_{\kern -1pt #2}}
\newcommand \diag [1]  {\Delta_{#1}}
\newcommand \Con [1]   {\textup{Con}(#1)}
\newcommand \relTheta {\mathrel{\Theta}}
\newcommand \ideal [1] {\mathord{\downarrow}#1}
\newcommand \filter [1] {\mathord{\uparrow}#1}
\newcommand \length [1]   {\textup{length}(#1)}
\newcommand \wh [1] {\widehat{#1}}

\newcommand\red[1]{{\textcolor{red}{#1}}}

\begin{document}
\title[Slim patch lattices]
{Slim patch lattices as absolute retracts and maximal lattices}

\author[G.\ Cz\'edli]{G\'abor Cz\'edli}
\email{czedli@math.u-szeged.hu}
\urladdr{http://www.math.u-szeged.hu/~czedli/}
\address{ Bolyai Institute, University of Szeged, Hungary}

\begin{abstract} \emph{Patch lattices},  introduced by G.~Cz\'edli and E.\,T.\ Schmidt in 2013, are the building stones for slim (and so necessarily finite and planar)  semimodular lattices with respect to gluing.  \emph{Slim semimodular lattices} were introduced by G.\ Gr\"atzer and E.\ Knapp in 2007, and they have been intensively studied since then. Outside lattice theory, these lattices played the main role in adding a uniqueness part to the classical Jordan--H\"older theorem for groups by G.~Cz\'edli and E.\,T.\ Schmidt in 2011, and they also led to results in combinatorial geometry. 
In this paper, we prove that \emph{slim patch lattices} are exactly the \emph{absolute retracts} with more than two elements for  the category of slim semimodular lattices  with length-preserving lattice embeddings as morphisms. Also, slim patch lattices are the same as the \emph{maximal objects} $L$ in this category such that $|L|>2$. 
Furthermore, slim patch lattices are characterized as the   \emph{algebraically closed lattices} $L$ in this category such that $|L|>2$.  Finally, we prove that if we consider $\{0,1\}$-preserving lattice homomorphisms rather than length-preserving ones, then the absolute retracts for the class of slim semimodular lattices are the at most 4-element boolean lattices.
\end{abstract}

\thanks{This research of the first author was supported by the National Research, Development and Innovation Fund of Hungary under funding scheme K 134851.}

%\dedicatory{Dedicated??? } 



\date{\datum\hfill{{Hint: check the author's website for possible updates}}}

\maketitle

\section{Introduction}\label{sect:intro}
\subsection{Goal}
Postponing the definitions to Subsection~\ref{subsect:defsurv}, we formulate our goal as follows. We intend to  characterize slim patch lattices among slim semimodular lattices as absolute retracts and also as maximal lattices. Theorem~\ref{thmmain} in Subsection~\ref{subsect:result} indicates that this is possible, provided we turn the class of slim semimodular lattices into a category with appropriately chosen morphisms and we disregard the singleton lattice and the two-element lattice. The present paper continues the research started in Cz\'edli and Molkhasi~\cite{czgmolkhasi}, where all lattice homomorphisms among slim semimodular lattices were allowed.

\subsection{Outline} In addition to giving a short historical survey and presenting our motivations, Subsection~\ref{subsect:defsurv} gives most of the definitions that are needed to state our main result, Theorem \ref{thmmain}, in Subsection \ref{subsect:result}. A related result and two corollaries are also formulated in Subsection \ref{subsect:result}. In Section \ref{sect:proof}, we prove the results.

\subsection{Definitions and a mini-survey}\label{subsect:defsurv}
All lattices in the paper are assumed to be \emph{finite} even where this is not emphasized. For (a finite) lattice $L$, the 
set of non-zero \emph{join-irreducible elements} and that of non-unit \emph{meet-irreducible elements} will be denoted by $\Jir L$ and $\Mir L$, respectively.
They are posets (that is, partially ordered sets) with respect to the order inherited from $L$.
Following  Cz\'edli and Schmidt \cite{czgschtJH}, we say that a lattice $L$ is \emph{slim} if it is finite and $\Jir L$ is the union of two chains. If $x\wedge y\prec x\Rightarrow y\prec x\vee y$ for all $x,y\in L$, then $L$ is \emph{semimodular}. 
The intensive study of planar semimodular lattices began with  Gr\"atzer and Knapp \cite{gratzerknapp1,gratzerknapp3}. For \emph{these} lattices, our definition of slimness is equivalent to their original one: a  planar semimodular lattice is slim if and only if the five-element modular lattice $M_3$ with three atoms is not a cover-preserving sublattice of $L$. Since each planar semimodular lattice is easily reduced to a slim semimodular lattice by  Gr\"atzer and Knapp \cite{gratzerknapp1}, slim semimodular lattices play a distinguished role among planar semimodular lattices. 

By  Lemma 2.2 of Cz\'edli and Schmidt \cite{czgschtJH}, slim lattices are \emph{planar}. (Since this is not so if the original definition of slimness from Gr\"atzer and Knapp \cite{gratzerknapp1} is used, the term ``slim planar semimodular lattice" also occurs in the literature.)

The original importance of slim semimodular lattices in \emph{lattice theory} is explained by their role in studying the congruence lattices of finite lattices; see Gr\"atzer and Knapp \cite{gratzerknapp1,gratzerknapp3} together with the book chapter Cz\'edli and Gr\"atzer~\cite{czgggltsta} and its references. Also, see Cz\'edli~\cite{czgqplanar} for a connection between these lattices and a variant of planarity of bounded posets. Finally, see  Cz\'edli \cite{czglamps}, Cz\'edli and Gr\"atzer~\cite{czgginprepar}, and their references for recent developments. 

Notably,  slim semimodular lattices have already found applications \emph{outside lattice theory}. First, they played a crucial role in generalizing the classical Jordan--H\"older theorem for \emph{groups} in Gr\"atzer and Nation~\cite{gr-nation} and  Cz\'edli and Schmidt~\cite{czgschtJH}. 
%5
Second, these lattices
led to new results in (combinatorial and convex) \emph{geometry}; see Adaricheva and Bolat~\cite{adaribolat}, Adaricheva and Cz\'edli~\cite{adariczg},
Cz\'edli~\cite{czgcircles}, Cz\'edli and Kurusa~\cite{czgkurusa}, and the references given in \cite{czgkurusa}. 
This connection is due to the canonical correspondence between slim semimodular lattices and (combinatorial) convex geometries of convex dimension at most 2; see Propositions 2.1 and 7.3 and Lemma 7.4 in Cz\'edli~\cite{czgcoord}.
Third, these lattices gave rise to interesting \emph{enumerative combinatorial} questions in several papers. For example, even the famous mathematical constant $e=\lim_{n\to\infty}(1+1/n)^n\approx 2.718\,2818$ appeared in a lattice theoretical and combinatorial paper; see Cz\'edli, D\'ek\'any,  Gyenizse, and Kulin~\cite{czgdgyk}.
Fourth, some connection between 
these lattices and finite model theory has recently been found in Cz\'edli~\cite{czgaxiombipart}. 
Fifth (and least), a computer game was developed based on these lattices; see  Cz\'edli and Makay~\cite{czgmakay}.

Next, we recall the following concept from Cz\'edli and Schmidt~\cite{czgschtpatchwork}; in a slightly modified form that needs less preparation. (The original definition will be given later in \eqref{pbx:pLrdmGhszrs}.)
An element of a lattice $x\in L$ is said to be \emph{doubly irreducible} if it has exactly one lower cover and exactly one (upper) cover. In other words, if $x\in \Jir L\cap\Mir L$.

\begin{definition}\label{def:slimpatch}
A slim semimodular lattice is a \emph{slim patch lattice} if it has exactly two doubly irreducible elements, these two elements are coatoms, and their meet is the smallest element of the lattice.
\end{definition}

For example, each of the three lattices drawn in Figure~\ref{fig1} is a slim patch lattice. (Their doubly irreducible elements are pentagon-shaped.)
By definition, a slim patch lattice consists of at least four elements. We proved in Cz\'edli and Schmidt~\cite{czgschtpatchwork} that every slim semimodular lattice can be obtained from slim patch lattices by gluing them together. In this sense, slim patch lattices are the ``small building stones'' among slim semimodular lattices. On  the other hand, slim patch lattices are ``large enough'' in the sense that each slim semimodular lattice $L$ can be embedded into a slim patch lattice $K$; for example, such a $K$ is constructed in  Cz\'edli and Molkhasi~\cite[Figure 2]{czgmolkhasi}. It will appear from our main result that patch lattices are ``maximally large'' in some sense.

Next, assume that 
\begin{equation}
\parbox{7.2cm}
{$\alg C$ is a (concrete) category that consists of some lattices as objects and each morphism of 
$\alg C$ is a lattice homomorphism;}
\label{pbx:CKTgr}
\end{equation}
we do not require that all lattice homomorphisms among
the objects of  $\alg C$  are  morphisms in $\alg C$. 
Using that every singleton subset of a lattice is a sublattice, it follows easily that
\begin{equation}
\parbox{8.3cm}{the monomorphisms of $\alg C$ given in \eqref{pbx:CKTgr} are lattice embeddings, that is, injective lattice homomorphisms.}
\label{eq:monoinj}
\end{equation} 
For lattices $L, K\in\alg C$, we say that $L$ is a \emph{retract} of $K$ in the category $\alg C$ if there is a morphism $\iota\colon L\to K$ in $\alg C$ and a morphism $\rho\colon K\to L$ in $\alg C$ such that $\rho\circ \iota$ is the identity morphism $\id L$ of $L$. 
Here, by \eqref{pbx:CKTgr},  $\iota$ and $\rho$ are lattice homomorphisms;  note that  we compose them  from right to left, that is, 
$(\rho\circ \iota)(x)=\rho(\iota(x))$. Note also that $\rho\circ \iota=\id L$ and \eqref{eq:monoinj}
imply that $\iota$ is a lattice embedding and it is a monomorphism in $\alg C$,  and $\rho$ is a surjective (in other words, an onto) map. The morphism $\rho$ above is called a \emph{retraction} of $\iota$. 

\begin{definition}\label{def:absretr} Let $\alg C$ be as in \eqref{pbx:CKTgr}.
A lattice $L\in\alg C$ is an \emph{absolute retract} for $\alg C$ if for every $K\in \alg C$ and every monomorphism $\iota\colon L\to K$, there exists a morphism $\rho\colon K\to L$ in $\alg C$ such that $\rho\circ\iota=\id L$. In other words, $L\in \alg C$ is an absolute retract for $\alg C$ if every monomorphism of $\alg C$ with domain $L$ has a retraction in $\alg C$. 
\end{definition}

This well-known concept for classes of algebras emerged in Reinhold~\cite{reinhold} in 1946.

\begin{definition}\label{def:maxobj}
Let $\alg C$ be as in  \eqref{pbx:CKTgr}.  We say that a lattice $L\in \alg C$ is a 
\emph{maximal object} of $\alg C$ if every monomorphism $L\to K$ of $\alg C$ is an isomorphism.
\end{definition}

It is quite rare that $\alg C$ has a maximal object. 
For a lattice $L$, an \emph{equation} in $L$ is a formal expression
\begin{equation}
p(a_1,\dots,a_m, x_1,\dots, x_n)\approx q(a_1,\dots,a_m, x_1,\dots, x_n)
\label{eq:eqdF}
\end{equation}
where $m\in\Nnul=\set{0,1,2,\dots}$, $n\in\Nplu=\Nnul\setminus\set 0$,
$p$ and $q$ are $(m+n)$-ary lattice terms,  the \emph{parameters} (also know as \emph{coefficients}) $a_1,\dots,a_m$ are in $L$, and 
$x_1,\dots, x_n$ are the \emph{unknowns} of \eqref{eq:eqdF}. If $\mu\colon L\to K$ is a lattice homomorphism, then the \emph{$\mu$-image} of equation \eqref{eq:eqdF} is the equation
\begin{equation*}
p(\mu(a_1),\dots,\mu(a_m), x_1,\dots, x_n)\approx q(\mu(a_1),\dots,\mu(a_m), x_1,\dots, x_n)
\end{equation*}
in $K$. For a set $\Sigma$ of equations in $L$, we let $\mu(\Sigma):=\set{\mu(e):e\in \Sigma}$. The following definition is taken from Schmid~\cite{schmid}, and it was used later in Molkhasi \cite{molkhasi16,molkhasi18b,molkhasi20}. 

\begin{definition}\label{def:algclose}
Let $\alg C$ be as in  \eqref{pbx:CKTgr}.  We say that a lattice $L\in \alg C$ is \emph{strongly algebraically closed in $\alg C$} if for any set $\Sigma$ of equations in $L$ and any monomorphism $\iota\colon L\to K$ in $\alg C$, if $\iota(\Sigma)$ has a solution in $K$, then $\Sigma$ has a solution in $L$. If we replace ``any set $\Sigma$'' by ``any finite set $\Sigma$'', then we obtain the concept of an \emph{algebraically closed} $L\in \alg C$.
\end{definition}

From now on, 
\begin{equation}
\parbox{7.2cm}{let $\sps$ denote the category of slim semimodular lattices with all lattice homomorphisms.}
\label{pbx:sScTrzrVrbRstR}
\end{equation}
Cz\'edli and Molkhasi~\cite{czgmolkhasi} proved that for a lattice $L\in \sps$, the following four conditions are equivalent: (1) $L$ is algebraically closed in $\sps$, (2) $L$ is strongly algebraically closed in $\sps$, (3) $L$ is  an absolute retract for $\sps$, and (4) $L$  is the singleton lattice. In addition to the importance of patch lattices, this result is also one of our motivations here.

\subsection{The results of the paper}\label{subsect:result}
A semimodular lattice $L$ is finite by definition, whence it has $0=0_L$ and $1=1_L$. For lattices $L$ and $K$ with 0 and 1, a lattice homomorphism $\phi\colon L\to K$ is a \emph{\zo{} homomorphism} if $\phi(0)=0$ and $\phi(1)=1$. If, in addition, $\phi$ is injective, then $\phi$ is a \emph{\zo{} embedding}. 
Also, for finite lattices $L$ and $K$, $\phi\colon L\to K$ is  said to be a \emph{\lemb{}} if $\phi$ is  a \zo{} embedding such that for all $x,y\in L$, if
$y$ covers $x$ in $L$ then $\phi(y)$ covers $\phi(x)$ in $K$. Equivalently, a \lemb{} is a \zo{} homomorphism $\phi\colon L\to K$ such that, for all $x,y\in L$, $x\prec y$ if and only if $\phi(x)\prec \phi(y)$. (Remember that $L$ and $K$ are finite by our general convention for this paper.)
Using that finite semimodular lattices satisfy the Jordan--H\"older condition, it is easy to see that if there is a \lemb{} $L\to K$, then $L$ and $K$ are of the same length. 
Let us emphasize that \zo{} homomorphisms and \lemb{}s are lattice-homomorphisms. To formulate our results, 
we need the following two categories.
\begin{align}
\parbox{8cm}{Let $\oisps$ denote the category of slim semimodular lattices  with \zo{} homomorphisms.}\label{pbx:sTrVsha}\\
\parbox{8cm}{Let $\lensps$ denote the category of slim semimodular lattices  with \lemb{}s.}\label{pbx:sTrVshb}
\end{align}
With more details, \eqref{pbx:sTrVshb} says that the objects of $\lensps$ are the slim semimodular lattices and the morphisms of $\lensps$ are the \lemb{s} among these lattices, and analogously for \eqref{pbx:sTrVsha}.

Now, based on Definitions \ref{def:slimpatch}--\ref{def:maxobj}
and notation \eqref{pbx:sTrVshb}, we are in the position to formulate the main result of the paper.

\begin{theorem}[Main Theorem]\label{thmmain}
For a slim semimodular lattice $L$,  the following three conditions are equivalent.
\begin{enumerate}[\upshape \qquad ({M}1)]
\item\label{thmmaina} $L$ is an absolute retract for $\lensps$.
\item\label{thmmainb} $L$ is a maximal object of $\lensps$.
\item\label{thmmainc} $L$ is a slim patch lattice or $|L|\leq 2$.
\end{enumerate}
\end{theorem}

This theorem  clearly yields the following corollary, which explains the title of the paper. Let $\lensps^{\geq 3}$ denote the full subcategory of $\lensps$ consisting of at least three-element slim semimodular lattices and all $\lensps$-morphisms among them.

\begin{corollary}\label{cor:xPlTl}
In  $\lensps^{\geq 3}$, slim patch lattices are characterized as absolute retracts. Also,   slim patch lattices are characterized as the maximal objects of  $\lensps^{\geq 3}$.
\end{corollary}

It is not rare that a class of ``important objects'' in algebra (and in some other fields of mathematics) has a category theoretical characterization. Corollary~\ref{cor:xPlTl} gives two such  characterizations of the class of slim patch lattices. Hence, in addition to the original motivation of introducing these lattices in Cz\'edli and Schmidt~\cite{czgschtpatchwork}, Corollary~\ref{cor:xPlTl} is another sign that slim patch lattices deserve attention.
So is the following corollary, which is based on Definitions \ref{def:slimpatch} and \ref{def:algclose};  it will be proved in Section~\ref{sect:proof}.

\begin{corollary}\label{cor:PclSd} For a slim semimodular lattice $L$, the following three conditions are equivalent.
\begin{enumeratei}
\item\label{cor:PclSda} $L$ is a slim patch lattice or $|L|\leq 2$.
\item\label{cor:PclSdb} $L$ is algebraically closed in $\lensps$.
\item\label{cor:PclSdc} $L$ is strongly algebraically closed in $\lensps$.
\end{enumeratei}
\end{corollary}

Next, we turn our attention to  category $\oisps$; see \eqref{pbx:sTrVsha}. It is not hard to see that there is no maximal object in $\oisps$. (For example, this will prompt follow from \eqref{pbx:frkXtSpS}.) The counterpart of the Main Theorem for this category is the following. 

\begin{proposition}\label{prop:wRw}
Let $L$ be a slim semimodular lattice. Then $L$ is an absolute retract for  $\oisps$ if and only if $L$ is an at most $4$-element boolean lattice.
\end{proposition}

Similarly to Corollary~\ref{cor:PclSd}, we have the following statement.

\begin{corollary}\label{cor:lGfnG} For a slim semimodular lattice $L$, the following three conditions are equivalent.
\begin{enumeratei}
\item\label{cor:PclSda} $L$ is algebraically closed in $\oisps$.
\item\label{cor:PclSdb} $L$ is strongly algebraically closed in $\oisps$.
\item\label{cor:PclSdc} $L$ is  an at most $4$-element boolean lattice.
\end{enumeratei}
\end{corollary}

\section{Proofs}\label{sect:proof}
Whenever we deal with a slim semimodular lattice, we always assume that a \emph{planar diagram} of this lattice \emph{is fixed}. Some of the concepts we are going to use depend on how this diagram is chosen but this will not cause any trouble. Below, for the sake of our proofs, we recall some concepts and statements from earlier papers. These concepts are also given in the book chapter Cz\'edli and Gr\"atzer \cite{czgggltsta}.

\begin{figure}[ht]
\centerline
{\includegraphics[width=\textwidth]{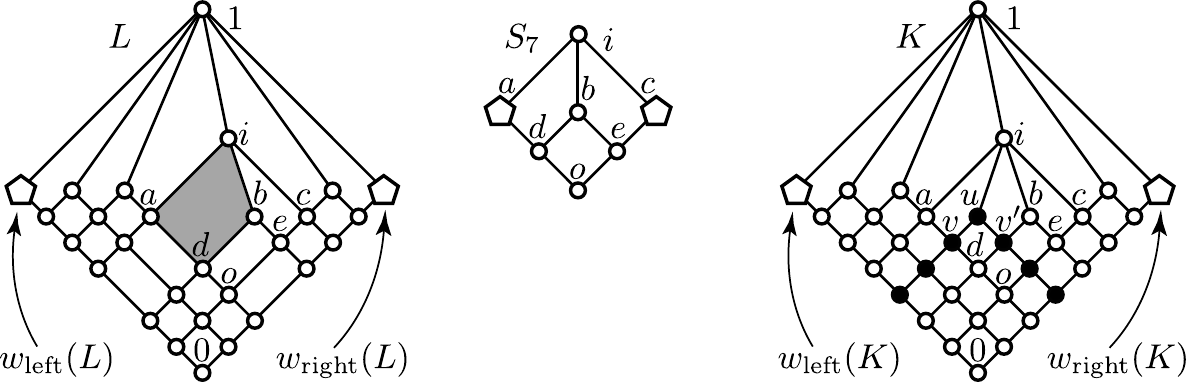}}      %[scale=0.93]{czgamfig1}}
\caption{Three slim patch lattices}\label{fig1}
\end{figure}

A cover-preserving four-element boolean sublattice of a slim semimodular lattice $L$ is called a \emph{$4$-cell}. Given a $4$-cell $C$ of this $L$, we can perform
a \emph{fork extension} of $L$ at the $4$-cell $C$ by \emph{adding a fork} to $C$
as it is shown in Figure 5 of Cz\'edli and Schmidt \cite{czgschtvisual}; this is also shown here in Figure~\ref{fig1}, where we add a fork to the grey-filled 4-cell of $L$ to obtain $K$. (The new elements, that is, the elements of $K\setminus L$, are black-filled.)  A fork extension yields  a proper extension of $L$, that is, $L$ is a proper sublattice of any of its fork extensions. (By a \emph{proper} sublattice we mean a sublattice distinct from the original lattice.) We know from Cz\'edli and Schmidt \cite[Theorem 11]{czgschtvisual} that
\begin{equation}
\parbox{7cm}{a fork extension of a slim semimodular lattice is again a slim semimodular lattice.}
\label{pbx:frkXtSpS}
\end{equation}

For a slim semimodular lattice $L$, keeping in mind that it is planar and a planar diagram of $L$ is fixed, the  \emph{left boundary chain} and the \emph{right boundary chain} of $L$ are denoted by $\lbound L$ and $\rbound L$, respectively. The union of these two chains is the \emph{boundary} of $L$; it is denoted by $\Bnd L$. 
The elements of $\Bnd L$ and the edges among these elements form a polygon in the plane, the \emph{boundary polygon} of (the fixed diagram of)  $L$. 
Following  Gr\"atzer and Knapp~\cite{gratzerknapp3}, a slim semimodular lattice $L$ is a \emph{slim rectangular lattice} if $\lbound L$ has exactly one doubly irreducible element, denoted by $\lc L$, $\rbound L$ has exactly one doubly irreducible element, denoted by $\rc L$, and these two elements are complementary, that is,  
\begin{equation}
\lc L\wedge \rc L=0\quad\text{ and }\quad \lc L\vee \rc L=1.
\label{eqRdgzmDtmKmL}
\end{equation}
Note that the  original definition of slim patch lattices in Cz\'edli and Schmidt~\cite{czgschtpatchwork} is the following:
\begin{equation}
\parbox{8.8cm}{a lattice $L$ is a \emph{slim patch lattice} if it is a slim rectangular lattice such that $\lc L$ and $\rc L$ are coatoms.}
\label{pbx:pLrdmGhszrs}
\end{equation}
The doubly irreducible coatoms of each of the three slim patch lattices in Figure~\ref{fig1} are the pentagon-shaped elements. 

For $u\in L$, the \emph{ideal} $\set{x\in L: x\leq u}$ and the filter
 $\set{x\in L: x\geq u}$ are denoted by $\ideal u$ and $\filter u$, respectively.  
We know from Lemmas 3 and 4 of Gr\"atzer and Knapp~\cite{gratzerknapp3} that for any rectangular lattice $L$,
\begin{align}
&\text{$\ideal{\lc L}$, $\ideal{\rc L}$, $\filter{\lc L}$, and $\filter{\rc L}$ are chains,}\label{eq:cskRbnSDtTvna}\\
&\text{$\filter{\lc L}\setminus\set 1\subseteq \Mir L$, \quad $\filter{\rc L}\setminus\set 1\subseteq \Mir L$,}\label{eq:cskRbnSDtTvnb}
\\
&\text{$\ideal{\lc L}\setminus\set 0\subseteq \Jir L$, \ and \ $\ideal{\rc L}\setminus\set 0\subseteq \Jir L$.}\label{eq:cskRbnSDtTvnc}\end{align}

The direct product of two non-singleton chains is a \emph{grid}. Grids are distributive slim rectangular lattices. We know from (the last sentence of) Theorem 11 and  Lemma 22 in \cite{czgschtvisual} that 
\begin{equation}
\parbox{9cm}{$L$ is a
\emph{slim rectangular} lattice if and only if it can be obtained from a  grid by adding forks, one by one, in a finite (possibly zero) number of steps.}
\label{pbx:sRctGlr}
\end{equation}
For $x\in \Jir L$ and $y\in \Mir L$, the unique lower cover of $x$ and the unique cover of $y$ are denoted by $\lcov x$ and $\ucov y$, respectively. 
Following Cz\'edli and Schmidt~\cite{czgschtvisual} again, a \emph{corner} of a slim semimodular lattice $L$ is a doubly irreducible element $u$ (necessarily on the boundary of $L$) such that $\ucov u$ covers exactly two elements and $\lcov u$ is covered by exactly two elements. Note that $\lc L$ and $\rc L$ defined 
in the paragraph preceding \eqref{pbx:pLrdmGhszrs} need not be corners; they are only \emph{weak corners} in the sense of in Cz\'edli and Schmidt~\cite{czgschtvisual}. 
We know from Lemma 21 of \cite{czgschtvisual} that 
\begin{equation}
\parbox{9cm}{a lattice $L$ is a slim semimodular lattice if and only if $|L|\leq 2$ or $L$ can be obtained  from a slim rectangular lattice by removing finitely many corners, one by one.}
\label{pbx:sLcRGns}
\end{equation}
It is trivial by definitions that 
\begin{equation}
\parbox{8.3cm}{If $L,K\in\lensps$ and $L$ is obtained from $K$ by removing a corner of $K$, then the embedding $\iota\colon L\to K$ defined by $\iota(x)=x$ for all $x\in L$ is a monomorphism in $\lensps$.}
\label{pbx:rmCrsKcCs}
\end{equation}

The \emph{congruence lattice} of a lattice $L$ will be denoted by $\Con L$, and $\diag L$ will stand for the \emph{equality relation} $\set{(x,x): x\in L}$, which is the least element of $\Con L$. For $\Theta\in \Con L$ and $H\subseteq L$, the \emph{restriction}  $\set{(x,y)\in H^2: (x,y)\in \Theta}$  of $\Theta$ to $H$ will be denoted by $\restrict \Theta H$. For $\Theta \in \Con L$ and $u\in L$, we denote by $u/\Theta$ the \emph{$\Theta$-block} $\set{x\in L: (u,x)\in\Theta}$ of $u$. We often write $u\relTheta v$ instead of $(u,v)\in \Theta$. 

Armed with the tools and notations listed so far in this section, we are prepared for the proof of our main result.

\begin{proof}[Proof of Theorem \ref{thmmain}]
First, to show the implication \mref{thmmainb} $\Rightarrow$  \mref{thmmaina}, assume that $L$ is a maximal object for $\lensps$. Now if $\iota\colon L\to K$ is a morphism in $\lensps$, then $\iota$ is an isomorphism since $L$ is a maximal object. Hence, we can let $\rho:=\iota^{-1}\colon K\to L$, and $\rho\circ\iota=\id L$ is clear. Thus, $L$ is an absolute retract for $\lensps$, proving the implication \mref{thmmainb} $\Rightarrow$  \mref{thmmaina}.

We recall  from Gr\"atzer and Nation \cite{gr-nation} that
\begin{equation}
\parbox{7.6cm}{if $C$ is a maximal chain of a finite semimodular 
lattice $L$, then every congruence $\Theta$ of $L$ is determined by its restriction $\restrict\Theta C$ to $C$.}
\label{pbx:GNlmrM}
\end{equation}
Note that Gr\"atzer and Nation \cite{gr-nation} proved a  more general result by allowing ``finite length" instead of ``finite''. 

Next, to show the implication  \mref{thmmaina} $\Rightarrow$  \mref{thmmainb}, assume that $L\in\lensps$ is an absolute retract for $\lensps$, and let $\iota\colon L\to K$ be a monomorphism in $\lensps$. By \eqref{eq:monoinj}, $\iota$ is injective. Take a maximal chain $C$ in $L$. Then $\iota(C):=\set{\iota(x): x\in C}$ is a maximal chain in $K$ since $\iota$ is a length-preserving embedding.
Since $L$ is an absolute retract for $\lensps$, $\iota$ has a retraction $\rho\colon K\to L$ in $\lensps$. For later reference, let us mention that in the rest of our argument proving the implication  \mref{thmmaina} $\Rightarrow$  \mref{thmmainb},
\begin{equation}
\parbox{7.7cm}{to prove that our length-preserving embedding $\iota$ is an isomorphism, we only use that $\rho$ is a lattice homomorphism such that $\rho\circ \iota=\id L$, but we do not use that $\rho$ is a morphism in $\lensps$.}
\label{pbx:wRhgRsZkrTpsz}
\end{equation}
Let $\Theta\in \Con K$ denote the kernel of $\rho$. Since $\rho\circ\iota=\id L$ gives that, for every $\iota(x)\in\iota(C)$, $\rho(\iota(x))=x$, the restriction of $\rho$ to $\iota(C)$ is injective. 
Hence,  $\restrict {\Theta}{\rho(C)}=\diag{\rho(C)}$. Applying 
\eqref{pbx:GNlmrM}, we obtain that $\Theta=\diag K$. Hence, $\rho$ is injective. But it is also surjective since $\rho\circ\iota = \id L$. Thus, $\rho$ is a lattice isomorphism, whereby it has an inverse, $\rho^{-1}\colon L\to K$, which is also a lattice isomorphism. Using $\rho\circ \iota=\id L$, we obtain that
$\iota=\id K \circ \iota =(\rho^{-1}\circ \rho )\circ\iota= \rho^{-1}\circ (\rho \circ\iota)= \rho^{-1}\circ \id L=\rho^{-1}$, showing that $\iota$ is a lattice isomorphism. Thus, $L$ is a maximal object of $\lensps$, and we have proved that  \mref{thmmaina} $\Rightarrow$  \mref{thmmainb}.

Next, to prove  the implication \mref{thmmainb} $\Rightarrow$  \mref{thmmainc}, assume that $L$ is a maximal  a maximal object of $\lensps$ such that $|L|\geq 3$. It follows from \eqref{pbx:sLcRGns} and   \eqref{pbx:rmCrsKcCs} that $L$ is a slim rectangular lattice. Hence, $\lc L$ and $\rc L$ make sense. We claim that
\begin{equation}
\text{$\lc L$ and $\rc L$ are coatoms.}
\label{eq:txtZkYrsvNTtr}
\end{equation}
For the sake of contradiction, suppose that \eqref{eq:txtZkYrsvNTtr} fails and, say, $\lc L$ is not a coatom. Then $\filter {\lc L}$, which is a chain by \eqref{eq:cskRbnSDtTvna}, has at least three elements. Hence, there are unique elements $u,v\in \filter{\lc L}$ such that $u\prec v\prec 1$. Extend $L$ to a poset $K:=L\cup\set{d}$ so that $d\notin L$ and $u\prec d\prec 1$. In the diagram of $K$, we position $d$  to the left of $v$.  
Since $u\in \Mir L$ by \eqref{eq:cskRbnSDtTvnb},  we have that $u=\lcov d$ has exactly two covers in $K$. Hence, it follows from Proposition 10(i) of Cz\'edli and Schmidt~\cite{czgschtvisual} that $K\in \lensps$. Clearly, $\iota\colon L\to K$ defined by $\iota(x)=x$ for all $x\in L$ is a length-preserving embedding. However, $\iota$ is not an isomorphism since $|K|=|L\cup\set d|=|L|+1>|L|$. This contradict the assumption that $L$ is a maximal object of $\lensps$ and proves the implication  \mref{thmmainb} $\Rightarrow$  \mref{thmmainc}.

Finally, to prove the validity of \mref{thmmainc} $\Rightarrow$  \mref{thmmainb}, observe that the one-element lattice and the two-element lattice are trivially maximal objects of $\lensps$. So we assume that $L$ is a slim patch lattice, and we need to show that it is a maximal object of $\lensps$. 
In fact, it suffices to show that an isomorphic copy of $L$ is a maximal object in $\lensps$; this is why we can take the map $x\mapsto x$ instead of a more involved embedding below. 

For the sake of contradiction, suppose that $L$ is not a maximal object and take a slim semimodular lattice $K$ such that $L$ is a proper sublattice of $K$ and the map $L\to K$ defined by $x\mapsto x$ is a length-preserving embedding. In particular, we have that $0:=0_K=0_L$ and $1:=1_K=1_L$.  Fix a planar diagram of $K$, and pick an element $p\in K\setminus L$. 
Finite semimodular lattices satisfy the Jordan--H\"older chain condition (that is, any two maximal chains of them are of the same length), and  $\length K=\length L$. Hence, we conclude that $L$ is a cover-preserving sublattice of $K$; that is,  for all $x,y\in L$, if $x\prec_L y$ in $L$, then $x\prec_K y$ in $K$.  
Visually, this means that $p$ is not on any edge of $L$, and similarly for any other element of $K\setminus L$.
Therefore, if we remove the elements of $K\setminus L$ with all edges adjacent to them, then we  get a planar diagram of $L$; let this diagram be what we fix for $L$.  We know that  $L\subset K$. (As opposed to some other branches of mathematics,  ``$\subset$'' is the conjunction of ``$\subseteq$'' and ``$\neq$''.) 

By the classical Jordan curve theorem, the boundary polygon of $L$ divides the plane into three pairwise disjoint subsets:  an interior region,  an exterior region, and the (set of geometrical points on the) boundary polygon. The first two subsets are topologically open while the third one is closed.  
Since $p$ is not on any edge of $L$, it is not on the boundary polygon.  Therefore, $p$ is either in the interior region of the boundary polygon, or it is in the exterior region of the boundary polygon; these two possibilities need separate treatments.
  
First, assume that $p$ is in the interior region of the boundary polygon. Since this region  is divided into 4-cells by Lemma 4 of Gr\"atzer and Knapp~\cite{gratzerknapp1} and $p$ is not on any edge of $L$, the element $p$ is inside the topologically open region determined by a 4-cell $H=\set{b=u\wedge v, u, v, t=u\vee v}$ of $L$.  By Kelly and Rival \cite[Proposition 1.4]{kellyrival}, we obtain that $b< p<t$. Using the Jordan--H\"older chain condition and $b\prec u\prec t$, it follows that $b\prec p \prec t$.   Thus, $u$, $v$, and $p$ are three different covers of $b$,  which contradicts Lemma 8 of  Gr\"atzer and Knapp~\cite{gratzerknapp1}.  

Second, assume that $p$ is in the exterior region of the boundary polygon. Note that $0 <p < 1$. Take a maximal chain $C$ in $K$ that contains $p$.
Since $p$ is in the (topologically open) exterior region of the boundary polygon but $0\in C$  is not, there are consecutive elements $r\prec s$ of $C$ such that 
$s$ is in  the exterior region of the boundary polygon but $r$ is not. Using  planarity or, to be more precise,  Kelly and Rival~\cite[Lemma 1.2]{kellyrival}, we obtain that  $r\in\Bnd L$. In particular, $r\in L$. By left--right symmetry, we can assume that $r\in\lbound L$.
By the already mentioned Lemma 8 of  Gr\"atzer and Knapp~\cite{gratzerknapp1}, $r$ has at most two covers in $K$. Since $s\in K$ is a cover of $r$ and $r\neq 1$ yields that $r$ also has at least one cover in $L$, we obtain that $r$ has exactly one cover in $L$. That is, $r\in \Mir L$.
Since $L$ is a slim patch lattice, 
$\lbound L$ is the disjoint union of $\ideal{\lc L}\setminus\set{\lc L}$, $\set{\lc L}$, and $\set 1$.
If we had that $r\in \ideal{\lc L}\setminus\set{\lc L}$, then $r\in \Mir L$ would belong to $\Jir L$ by \eqref{eq:cskRbnSDtTvnc}, whence  $r$, $\lc L$, and $\rc L$ would be three different doubly irreducible elements of $L$, contradicting Definition ~\ref{def:slimpatch}. Hence, taking also $r\neq 1$ into account, we have that   $r=\lc L$. Thus, $\lc L\prec_K s\leq p<1$, which gives that $1$ does not cover $\lc L$ in $K$. This contradicts the facts that $\lc L$ is a coatom in $L$ and $K$ is a cover-preserving extension of $L$. 

Regardless of the position of $p$, we have obtained a contradiction. This yields the implication  \mref{thmmainc} $\Rightarrow$  \mref{thmmainb} and completes the proof of Theorem ~\ref{thmmain}.
\end{proof}

\begin{proof}[Proof of Proposition \ref{prop:wRw}]
To prove the ``only if'' part, assume that $L\in \oisps$ is an absolute retract for $\oisps$. We are going to show that $L$ is a maximal object of $\lensps$. Let $\iota\colon L\to K$ be a monomorphism in $\lensps$. By \eqref{eq:monoinj}, $\iota$ is a length-preserving embedding. 
Clearly, $\iota$ is \zo{}, so it is also a monomorphism in $\oisps$. 
Since we have assumed that $L$ is  an absolute retract for $\oisps$, there exists a \zo{} homomorphism $\rho\colon K\to L$ such that $\rho\circ\iota=\id L$. Applying \eqref{pbx:wRhgRsZkrTpsz}, it follows that $\iota$ is an isomorphism. 
This shows that $L$ is a maximal object of $\lensps$, as required.
Thus, we obtain from Theorem~\ref{thmmain} that $|L|\leq 2$ or $L$ is a slim patch lattice. We can assume that $L$ is a slim patch lattice since lattices with at most two elements are boolean. 
Then $|L|\geq 4$. 

For the sake of contradiction, suppose that $|L|\geq 5$. We know from \eqref{pbx:sRctGlr} that $L$ can be obtained from a grid $G$ by adding forks. When we add a fork to a slim rectangular lattice $R$, then $\lc R$ and $\rc R$ remain doubly irreducible and the lengths of the intervals $[\lc R, 1]$ and $|\rc R, 1]$ do not change.
Since $L$ is not only rectangular but it is a patch lattice, $\lc G$ and $\rc G$ are coatoms of $G$. This means that $G$ is the 4-element boolean lattice. Then, since $|G|=4<5\leq |R|$, it follows that at least one fork has been added to $G$ to obtain $L$. 
Thus, thinking of the last fork added, we obtain that   the lattice $S_7$ given in the middle of  Figure~\ref{fig1} is a cover-preserving sublattice of $L$. 

The elements of this $S_7$ will be denoted as in Figure~\ref{fig1}. Take the upper left 4-cell of this $S_7$; it is grey-filled on the left of  Figure~\ref{fig1}. Add a fork to $L$ to  obtain a new lattice denoted by $K$; see on the right of the figure.  The new meet-irreducible element is denoted by $u$, its lower covers are $v$ and $v'$, as it is shown in the figure. By \eqref{pbx:frkXtSpS}, $K\in \oisps$. Clearly, the embedding $\iota\colon L\to K$ defined by $x\mapsto x$ is a morphism in $\oisps$.  Since we have assumed that $L$ is an absolute retract for $\oisps$, $\iota$ has  a retraction $\rho\colon K\to L$ in $\oisps$.  That is, $\rho$ is a \zo{}-homomorphism such that $\rho\circ\iota=\id L$. In particular,  $\rho(x)=x$ for all $x\in L$. As in the previous proof, we let $\Theta=\ker \rho$. Observe that since $\rho(x)=x$ for all $x\in L$, 
\begin{equation}
\text{the restriction  $\Theta$ to $L$ is $\diag L$, that is, $\restrict\Theta L=\diag L$.}
\label{txtRswgRmtrMh}
\end{equation}

Since  $\rho(u)= (\rho\circ\iota)(\rho(u))=\rho((\iota\circ\rho)(u))  = \rho(\iota(\rho(u)))=\rho(\rho(u))$, we have that  $u\relTheta \rho(u)$.  Also,
$u\neq \rho(u)$ since $\rho(u)\in L$ but $u\notin L$. Depending on whether $u\not\leq \rho(u)$ or $u\not\geq \rho(u)$, we have that $u>u\wedge \rho(u)$ or $u<u\vee \rho(u)$. 
Since  $(u,\rho(u))\in \Theta$ gives that  $\set{u\wedge \rho(u),  u\vee \rho(u)}\subseteq  u/\Theta$, it follows that  $u$ is not a minimal element or not a maximal element of $u/\Theta$. Using that $u/\Theta$ is a convex subset of $K$ and taking into account that  $u$ covers or covered by exactly three elements, $i$, $v$, and $v'$, we obtain that at least one of $(u,i)$, $(v,u)$, and $(v',u)$ belongs to $\Theta$. This gives us three cases to consider; each of them leads to contradiction in a different way.

If $(u,i)\in\Theta$, then it follows immediately from Gr\"atzer's Swing Lemma, see his paper \cite{ggswinglemma} (alternatively, see   Cz\'edli, Gr\"atzer, and  Lakser \cite{czggghlswing} 
or Cz\'edli and Makay \cite{czgmakay} for secondary sources) 
that $(b,i)\in\Theta$, contradicting \eqref{txtRswgRmtrMh}. 
If $(v,u)\in \Theta$, then $(a,i)=(a\vee v,a\vee u)\in \Theta$ contradicts \eqref{txtRswgRmtrMh}.
If $(v',u)\in\Theta$, then $(b,i)=(b\vee v',b\vee u)\in\Theta$, contradicting \eqref{txtRswgRmtrMh} again.   Thus, $|L|\geq 5$ leads to a contradiction and it follows that $|L|=4$. Finally, a four-element patch lattice is boolean, and we have shown the ``only if'' part of Proposition~\ref{prop:wRw}.

Next, in order to prove the ``if'' part, assume that $L$ is a boolean lattice with at most four elements, $L\in \oisps$, and $\iota\colon L\to L'$ is a monomorphism in $\oisps$. That is, $\iota$ is a lattice embedding preserving $0$ and $1$. We need to find a morphism $\rho$ in $\oisps$ such that $\rho\circ\iota=\id L$. If $|L|=1$, then the preservation of 0 and 1 gives that $|L'|=1$, $\iota$ is an isomorphism, and we can let $\rho:=\iota^{-1}$. Hence, in the rest of the proof, it suffices to deal with the cases $|L|=2$ and $|L|=4$. Before doing so, let us recall from Gr\"atzer~\cite[Corollary 14]{ggapplczgschtseq} that
\begin{equation}
\parbox{7.9cm}{if $1\neq p\in \filter{\lc K}\cup\filter{\rc K}$ in
a slim rectangular lattice $K$, then $\ideal p$ is a prime ideal of $K$.}
\label{pbx:prMdlJgbbW}
\end{equation}

First, assume that $|L|=2$. Then $L=\set{0,1}$. We can assume that $|L'|>2$ since otherwise $\iota$ is an isomorphism and $\rho:=\iota^{-1}$ does the job. It follows from \eqref{pbx:sLcRGns} and \eqref{pbx:rmCrsKcCs} that there exists a slim
rectangular lattice $K$ and a monomorphism 
 $\iota'\colon L'\to K$ in $\oisps$. (In fact, $\iota'$ belongs even to $\lensps$.) By \eqref{eq:monoinj}, $\iota'$ is a lattice embedding.  We know from \eqref{pbx:prMdlJgbbW} that $I:=\ideal{\lc K}$ is a prime ideal of $K$. Hence, the map 
\[
\rho'\colon K\to L,\text{ defined by }
x\mapsto
 \begin{cases}
   0,&\text{ if }x\in I,\cr
   1,&\text{ if }x\in K\setminus I,
 \end{cases}
\]
is a lattice homomorphism. In fact, $\rho'$ is a morphism in $\oisps$. Let $\rho:=\rho'\circ\iota'$. Then $\rho$ is a map from $L'$ to $L$. Since both $\rho'$ and $\iota'$ are morphisms in $\oisps$, so is their product, $\rho$. By the same reason, $\rho\circ \iota\colon L\to L$ is again a morphism in $\oisps$. Since $L=\set{0,1}$, $\id L$ is the only  $L\to L$ morphism belonging to $\oisps$. Hence, $\rho\circ\iota=\id L$, showing that $\rho$ is a retraction of $\iota$. Therefore,  $L$ is an absolute retract for $\oisps$.

Second, assume that $|L|$ is the four-element boolean lattice and $\iota\colon L\to L'$ is a monomorphism in $\oisps$.
As in the previous case,  \eqref{eq:monoinj}, \eqref{pbx:sLcRGns}, and \eqref{pbx:rmCrsKcCs} yield that 
there exists  a slim rectangular lattice $K$ and a lattice embedding $\iota'\colon L'\to K$ such that $\iota'$ is a monomorphism in $\oisps$.
As previously, $I:=\ideal{\lc K}$ is a prime ideal of $K$, and so is $J:=\ideal{\rc K}$.  The atoms of $L$ will be denoted by $u$ and $v$, and we let $\wh u:=(\iota'\circ\iota)(u)=\iota'(\iota(u))$ and $\wh v:=(\iota'\circ\iota)(v)$. Of course, we have that $(\iota'\circ\iota)(0)=0$ and $(\iota'\circ\iota)(1)=1$ since we are in $\oisps$. Since $\iota'\circ\iota$ is a lattice embedding, $\wh u\wedge \wh v=0$, $\wh u\vee \wh v=1$, and $|\set{0,\wh u, \wh v, 1}|=4$. Since $I$ and $J$ are  prime ideals, we can observe that $\set{\wh u,\wh v}\not\subseteq I$ and $\set{\wh u,\wh v}\not \subseteq K\setminus I$, and analogously for $J$, since otherwise  $\wh u\vee \wh v=1$ or $\wh u\wedge \wh v=0$ would fail.   That is
\begin{equation}
|\set{\wh u,\wh v}\cap I|=1 \quad\text{ and }\quad|\set{\wh u,\wh v}\cap J|=1.
\label{eq:kTsjbbnWskTmHf}
\end{equation}
Thus, using that  $u$ and $v$ play  symmetrical roles, 
we can assume that $\wh u\in I$ but $\wh v\notin I$. It follows from \eqref{eqRdgzmDtmKmL} that $I\cap J=\set{0}$.
This fact and  $0\neq \wh u\in I$ give that $\wh u\notin J$. Combining this with \eqref{eq:kTsjbbnWskTmHf}, we obtain that $\wh v\in J$. So \eqref{eq:kTsjbbnWskTmHf} gives that 
\begin{equation}
\wh u\in I,\quad \wh u\notin J,\quad \wh v\in J,\quad\text{and}\quad \wh v\notin I. 
\label{eq:zdpVshtjK}
\end{equation}
Applying \eqref{eqRdgzmDtmKmL} again, we conclude easily that 
\begin{equation}
\lc K\in I,\,\text{  }\lc K\notin J,\,\text{  } \rc K \in J,\,\text{  }\text{and}\,\text{  } \rc K\notin I. 
\label{eq:jjTpVprLWv}
\end{equation}
Since $I$ is a prime ideal, the equivalence $\alpha$ 
with blocks $I$ and $K\setminus I$ is a congruence of $K$. Similarly, let $\beta$ be the congruence with blocks $J$ and $K\setminus J$. Let $\gamma:=\alpha\wedge \beta=\alpha\cap\beta\in\Con K$. Since each of $\alpha$ and $\beta$ has only two blocks, $\gamma$ has at most four blocks. 
It follows from \eqref{eq:jjTpVprLWv} that $0$, $1$, $\lc K$, and $\rc K$ are in different $\gamma$-blocks.  Hence $\gamma$ has exactly four blocks. Using that $\lc K$ and $\rc K$ are complementary, $K/\gamma$ is isomorphic to $L$. Hence, the map
\begin{equation}
\rho'\colon K\to L,\text{ defined by }x\mapsto
\begin{cases}
u,&\text{if }(x,\lc K)\in\gamma,\cr
v,&\text{if }(x,\rc K)\in\gamma,\cr
0,&\text{if }(x,0)\in\gamma,\cr
1,&\text{if }(x,1)\in\gamma.\cr
\end{cases}
\label{eq:sdgTlsgPlkSn}
\end{equation} 
is a lattice homomorphism and it is a morphism belonging to $\oisps$.
Comparing \eqref{eq:zdpVshtjK} and \eqref{eq:jjTpVprLWv}, we obtain that $(\wh u,\lc K)\in\gamma$ and $(\wh v,\rc K)\in\gamma$. 
Hence, it follows from \eqref{eq:sdgTlsgPlkSn} that $\rho'(\wh u)=u$,
$\rho'(\wh v)=v$. These two equalities and the fact that all the homomorphisms occur in the proof are morphisms in $\oisps$ imply that $\rho'\circ(\iota'\circ\iota)=\id L$. In other words,
 $(\rho'\circ\iota')\circ\iota=\id L$. Thus, with $\rho:=\rho'\circ\iota'$, which belongs to $\oisps$, we have that $\rho\circ \iota=\id L$. Since $\rho$ is an $L'\to L$ morphism in $\oisps$,
$L$ is an absolute retract for $\oisps$, as required. 
Hence, the ``if'' part holds and  the proof of Proposition~\ref{prop:wRw} is complete.
\end{proof}

Finally, we give a joint proof of the corollaries.

\begin{proof}[Proof of Corollaries \eqref{cor:PclSd} and \eqref{cor:lGfnG}] Compared to $\sps$ defined in \eqref{pbx:sScTrzrVrbRstR}, 
the categories $\oisps$ and $\lensps$ have some special properties. However, the proof of Proposition 1.1 of Cz\'edli and Molkhasi \cite{czgmolkhasi} does not use these properties. Thus, we conclude the validity of Corollaries \eqref{cor:PclSd} and \eqref{cor:lGfnG} from  \cite{czgmolkhasi}.
\end{proof}

\end{document}